\documentclass[10pt]{amsart}

\usepackage{times,amsmath,amsbsy,amssymb,amscd,mathrsfs}
\usepackage{slashbox}
\usepackage{graphicx,subfigure,epstopdf,wrapfig,chemarrow}
\usepackage{algorithm2e} 
\usepackage{multicol,multirow}
\usepackage{mathtools}
\usepackage[usenames,dvipsnames,svgnames,table]{xcolor}
\usepackage[numbered]{mcode}
\definecolor{myBlue}{rgb}{0.0,0.0,0.55}
\definecolor{green}{rgb}{0.0,0.7,0.2}
\usepackage[pdftex,colorlinks=true,citecolor=myBlue,linkcolor=myBlue]{hyperref}

\usepackage{comment,enumerate,multicol,xspace}

  \newcounter{mnote}
  \let\oldmarginpar\marginpar
    \renewcommand\marginpar[1]{\-\oldmarginpar[\raggedleft\footnotesize #1]%
    {\raggedright\footnotesize #1}}

\newtheorem{theorem}{Theorem}[section]
\newtheorem{lemma}[theorem]{Lemma}
\newtheorem{corollary}[theorem]{Corollary}

\newtheorem{example}[theorem]{Example}

\newtheorem{remark}[theorem]{Remark}
\newtheorem{alg}[theorem]{Algorithm}

\newcommand{\dd}{\,{\rm d}}

\newcommand{\bs}{\boldsymbol}

\newcommand{\csch}{{\rm csch}}

\newcommand{\vertiii}[1]{{\left\vert\kern-0.25ex\left\vert\kern-0.25ex\left\vert #1 
    \right\vert\kern-0.25ex\right\vert\kern-0.25ex\right\vert}}

\newcommand{\curl}{{\rm curl\,}}
\renewcommand{\div}{\operatorname{div}}
\newcommand{\grad}{{\rm grad\,}}

\usepackage[margins]{trackchanges}

\begin{document}
\title[MFEM for FHD model]{Mixed finite element methods for the ferrofluid model with magnetization paralleled to the magnetic field}
\author{Yongke Wu \and Xiaoping Xie}
\date{\today}
\thanks{Y. Wu was supported by the National Natural Science Foundation of China (11971094 and 11501088), and X. Xie was supported by the National Nature Science Foundation of China (12171340)}
\thanks{Y. Wu, School of Mathematical Sciences, University of Electronic Science and Technology of China, Chengdu 611731, China. Email: wuyongke1982@sina.com}
\thanks{X. Xie (Corresponding author), School of Mathematical, Sichuan University,  Chengdu 610064, China. Email: xpxie@scu.edu.cn}

%

\subjclass[2010]{
65N55;   
65F10;   
65N22;   
65N30;   
}

\begin{abstract}
Mixed finite element methods are considered for a ferrofluid flow model with magnetization paralleled to the magnetic field.  The  ferrofluid model is a coupled system of the Maxwell equations and the incompressible Navier-Stokes equations. 
By skillfully  introducing some new variables, the   model is  rewritten as  several  decoupled subsystems that can be   solved independently. 
Mixed finite element formulations are given  to discretize the decoupled systems  with   proper finite element spaces. Existence and uniqueness of the mixed finite element solutions are shown, and   optimal order error estimates  are obtained under some reasonable assumptions.  Numerical experiments confirm the theoretical results.  
\end{abstract}

\keywords{ferrofluid flow, decoupled system, mixed finite element method, error estimate}
\maketitle

\section{Introduction}
Ferrofluids are colloidal liquids consisting of non-conductive nanoscale ferromagnetic or ferrimagnetic particles suspended in carrier fluids, and have  extensive applications  in  many technology  and  biomedicine fields ~\cite{Pankhurst2003, Zahn2001}. 
There are two main ferrohydrodynamics (FHD) models which treat ferrofluids as homogeneous monophase fluids:   the Rosensweig's model ~\cite{Rosensweig1985,Rosensweig1987} and the  Shliomis' model~\cite{Shliomis1972Effective,Shliomis2002Ferrofluids}.  The main difference between these two  models lies in  that the former one considers the internal rotation of the nanoparticles, while the latter  deals with the rotation as a magnetic torque. We refer  to \cite{Amirat2008,Amirat2008GlobalR,Amirat2009Strong,Amirat2010UniqueR,Nochetto2019} for some  existence results of solutions to the two FHD models.

The  FHD models are coupled nonlinear systems of  the  Maxwell equations and the incompressible Navier-Stokes equations. 
There are limited works in   the literature on the numerical analysis  of the FHD models. In~\cite{Snyder2003Finite,Lavrova2006,Knobloch2010Numerical,Yoshikawa2010Numerical}, several numerical schemes were applied to solve  reduced  two dimensional  FHD models   where some nonlinear terms of the original models are dropped. 
  Nochetto et al.  \cite{Nochetto2015The} showed the energy stability of the  Rosensweig's   model, proposed an energy-stable numerical scheme using finite element methods, and gave the existence and  convergence of the numerical solutions.  Recently, Wu and Xie  \cite{YW2022} developed   a class of   energy-preserving mixed finite element methods for the Shliomis' FHD model,  and derived optimal error estimates for both the the
semi- and fully discrete schemes.
 We also note that in \cite{Zhang2021}   an unconditionally energy-stable fully discrete finite element method was presented for a two-phase FHD model. 


In this paper, we consider the Shliomis' FHD model with the assumption that the magnetization field is parallel to the magnetic field. Under this assumption, the magnetization equation in the Shliomis'   model degenerates to the Langevin magnetization law~\cite{Lavrova2006,Rosensweig1985,Rosensweig1987}.  We introduce some new variables to transform the  model   into  two main decoupled subsystems, i.e., a nonlinear elliptic equation and  the incompressible Navier-Stokes equations. We apply proper finite element spaces to discretize the nonlinear decoupled systems,  prove the existence and, under some reasonable assumptions,   uniqueness of the finite element solutions, and derive optimal error estimates. We also show that our scheme preserves the ferrofluids' nonconductive property  $\curl\bs H = 0$ exactly.

The rest of this paper is organized as follows. In section 2, we introduce several Sobolev spaces, give the governing equation of the FHD model with magnetization paralleled to magnetic field, reform the FHD model equivalently, and construct the weak formulations. In section 3, we recall the finite element spaces, show the existence and uniqueness of solutions for the constructed finite element methods, and give the optimal order error estimates. In section 4, some numerical experiments will be given to verify our theoretical results. 

\section{Preliminary}
In this section, we introduce several Sobolev spaces, give the governing equations of the FHD model with magnetization paralleled to magnetic field, derive  the equivalent decoupled systems, and present the weak formulations. 
\subsection{Sobolev spaces}
Let $\Omega \subset\mathbb R^d$ ($d = 2,3$) be a bounded and simply connected convex domain with Lipschitz boundary $\partial\Omega$, and $\bs n$ be the unit outward normal vector on $\partial\Omega$.

For any $p \geq 1$, we denote by $L^p(\Omega)$ the space of all power-$p$ integrable functions on $\Omega$ with norm $\|\cdot\|_{L^p}$. For any nonnegative integer $m$,   denote by $H^m(\Omega)$ the usual $m$-th order Sobolev space with norm $\|\cdot\|_m$  and semi-norm $|\cdot|_m$. In particular, $H^0(\Omega) = L^2(\Omega)$ denotes the space of square integrable functions on $\Omega$, with the inner product $(\cdot,\cdot)$ and norm $\|\cdot\|$. For the vector spaces $\bs H^m(\Omega):=(H^m(\Omega))^d$ and $\bs L^2(\Omega) := (L^2(\Omega))^d$, we use the same notations of norm, semi-norm and inner product as those for the scalar cases. 

We further introduce the Sobolev spaces
$$
\bs H(\curl): = \left\{
\begin{array}{ll}
\{\bs v \in (L^2(\Omega))^2:~~\curl\bs v \in L^2(\Omega)\} & \text{if } d = 2,\\
 \{\bs v \in (L^2(\Omega))^3:~~\curl\bs v \in (L^2(\Omega))^3\} & \text{if } d = 3
  \end{array}
\right.
$$
and
$$
\bs H(\div) := \{\bs v \in (L^2(\Omega))^d:~~\div\bs v \in L^2(\Omega)\},
$$
and set
\begin{align*}
&H_0^1(\Omega) := \{ v \in H^1(\Omega):~v = 0\text{ on }\partial\Omega\},\\
&\bs H_0(\curl) := \{\bs v \in \bs H(\curl):~\bs v\times \bs n = 0\text{ on }\partial\Omega\},\\
&\bs H_0(\div) := \{\bs v \in \bs H(\div):~\bs v \cdot \bs n = 0\text{ on }\partial\Omega\}, \\
& L_0^2(\Omega) := \{ v\in L^2(\Omega):~\int_\Omega v\dd \Omega = 0\}, 
\end{align*}
where 
\begin{align*}
& \curl \bs v: = \left\{\begin{array}{ll}
 \partial_x v_2 - \partial_y v_1 & \text{if }\bs v = ( v_1,~ v_2)^\intercal, \\
 (\partial_y v_3 - \partial_z v_2,~\partial_z v_1 - \partial_x v_3,~\partial_x v_2 - \partial_y v_1)^\intercal & \text{if }\bs v = (v_1,~v_2,~v_3)^\intercal. 
 \end{array}
\right.
\\
 & \div \bs v := \left\{
 \begin{array}{ll}
 \partial_x v_1 + \partial_y v_2 & \text{if }\bs v = (v_1,v_2)^\intercal, \\
 \partial_x v_1 + \partial_y v_2 + \partial_z v_3 & \text{if } \bs v = (v_1,v_2,v_3)^\intercal,
 \end{array}
 \right.
\end{align*}  
and $(x,y)$ and $(x,y,z)$ are  the Cartesian coordinates in two and three dimensions, respectively.

For any Sobolev space $S$ with norm $\|\cdot\|_S$, we use $S^\prime$ to denote the dual space of $S$, and  $\langle\cdot,\cdot\rangle$ to denote the dual product between $S^\prime$ and $S$. For any $f \in S^\prime$, the operator norm of $f$ is defined as $\|f\|_{S^\prime} = \sup\limits_{0\neq v \in S}\frac{\langle f,v\rangle}{\|v\|_S}$.

\subsection{Governing equations of the ferrofluid flow}
We consider the domain $\Omega$ filled with ferrofluid flow. On the macroscopic level, the mathematical model for describing the interactions between magnetic fields and ferrofluids consists of the Maxwell's equations and the Navier-Stokes equations~\cite{Rosensweig1987,Rosensweig1985,Shliomis1972Effective,Shliomis2002Ferrofluids}. Since the ferrofluid flow is nonconductive, the corresponding  Maxwell's equations read as:
\begin{align}
\label{eq:maxwell1} & \curl\bs H  = 0\qquad\qquad\qquad\,\text{in}\quad\Omega,\\
\label{eq:maxwell2} & \div\bs B  = \div\bs H_e\qquad\qquad\text{in}\quad\Omega,
\end{align}
with the magnetic field $\bs H$ and the magnetic induction $\bs B$ satisfying the relation
\begin{equation}
\label{eq:rel-H-B} 
\bs B = \mu_{0} (\bs H + \bs M)\qquad\qquad\text{in}\quad\Omega,
\end{equation}
where $\mu_{0}>0$ is the permeability constant, $\bs M$ is the magnetization, and $\bs H_e$ is the known external magnetic field that satisfies $\bs H_e\cdot\bs n = 0$ on $\partial\Omega$. 

Under the assumption that the magnetization $\bs M$ of the ferrofluid flow is parallel to the magnetic field $\bs H$, it follows the nonlinear Langevin magnetization law~\cite{Rosensweig1987,Rosensweig1985,Lavrova2006}, i.e.,
\begin{equation}
\label{eq:langevin}
\bs M(\bs H) = M_{s}\left(\coth(\gamma H)- \frac{1}{\gamma H} \right)\frac{\bs H}{H},
\end{equation}
with the saturation magnetization $M_{s}>0$, the Lagevin parameter $\gamma= 3\chi_{0}/  M_{s}$, the initial susceptibility $\chi_{0}>0$, $H := |\bs H| = (\bs H \cdot\bs H)^{1/2} > 0$, and $\coth(x) = \frac{e^x+e^{-x}}{e^x-e^{-x}}$. 

The hydrodynamic properties of the ferrofluid flow are described by the incompressible    Navier-Stokes equations
\begin{align}
\label{eq:NS1} & \rho(\bs u\cdot \nabla)\bs u - \eta\Delta\bs u + \nabla p  = \bs f + \mu_{0} (\bs M\cdot\nabla) \bs H\qquad\quad\text{in}\quad\Omega,\\
\label{eq:NS2} & \div\bs u  = 0\,\,\,\,\,\quad\qquad\qquad\qquad\qquad\qquad\qquad\qquad\qquad\text{in}\quad\Omega,
\end{align}
where $\rho$ denotes the fluid density, $\bs u$   the velocity field of the flow, $\eta$  the dynamic viscosity, and $\bs f$   the  known volume force.

We consider the following homogenous boundary conditions for equation \eqref{eq:maxwell1}-\eqref{eq:NS2}: 
\begin{equation}
\label{eq:bd} \bs u = 0\quad\text{  and  }\quad \bs H \times \bs n = 0\qquad\text{on}\quad\partial\Omega.
\end{equation}

Using the fact that
$$
(\bs H \cdot\nabla)\bs H = \frac{1}{2}\nabla H^2 - \bs H \times (\curl\bs H) = \frac{1}{2} \nabla H^2
$$
and the Langevin magnetization law \eqref{eq:langevin}, we have
\begin{equation}
\label{eq:Kevinforce}
(\bs M\cdot\nabla)\bs H = \frac{M_s}{H}\left( \coth(\gamma H) - \frac{1}{\gamma H} \right) \frac{1}{2}\nabla H^2 = \frac{M_s}{\gamma} \nabla \ln \frac{\sinh(\gamma H)}{H},
\end{equation}
where $\sinh(x) = \frac{e^x-e^{-x}}{2}$.  Denote 
\begin{equation}
\label{eq:beta-def}
\beta(x): = \frac{M_s}{\gamma} \ln\frac{\sinh(\gamma x)}{x},
\end{equation}
and introduce two variables
\begin{equation}\label{eq:psi}
\psi :=\beta(H) 
\end{equation}
and
\begin{equation}\label{eq:p}
\tilde p: = p - \mu_0\psi,
\end{equation}
then equation \eqref{eq:NS1} becomes
\begin{equation}
\label{eq:NS1-n}
\rho(\bs u\cdot\nabla)\bs u - \eta\Delta\bs u + \nabla \tilde p = \bs f\qquad\qquad\qquad\text{in }\Omega.
\end{equation}
Equation \eqref{eq:maxwell1} and the 
boundary condition $\bs H \times \bs n|_{\partial\Omega} = 0$ in \eqref{eq:bd} imply  that (cf. \cite{Arnold;Falk;Winther2006}) there exists $\phi \in H_{0}^{1}(\Omega)$ with
\begin{equation}
\label{eq:Hphi} \bs H = \nabla\phi,
\end{equation}
then combining \eqref{eq:maxwell2}, \eqref{eq:rel-H-B}  and \eqref{eq:langevin} leads to 
\begin{equation}
\label{eq:H} \nabla\cdot\left(\alpha(|\nabla\phi|)\nabla\phi\right) = \frac{1}{\mu_0}\div\bs H_e =: g\qquad\text{in}~~\Omega,
\end{equation}
with 
\begin{equation}\label{eq:alpha-def}
\alpha(x) := 1+\frac{M_{s}}{x}\left(\coth(\gamma x) - \frac{1}{\gamma x}\right).
\end{equation}

The above equivalence transformation shows that the FHD model \eqref{eq:maxwell1} - \eqref{eq:NS2} with the boundary conditions  \eqref{eq:bd} can be equivalently written as follows: Find $\phi \in H_0^1(\Omega)$, $\bs H \in\bs H_0(\curl)$, $\bs M \in \bs H_0(\curl)$, $\bs u \in \bs H_0^1(\Omega)$, $\tilde p \in L_0^2(\Omega)$, $\psi \in L_0^2(\Omega)$, and $p \in L_0^2(\Omega)$ such that 
\begin{equation}
\label{eq:eq-sys}
\left\{
\begin{array}{ll}
\nabla\cdot(\alpha(|\nabla\phi|)\nabla\phi) = g & \text{in }\Omega,\\
\bs H - \nabla\phi = 0 & \text{in }\Omega, \\
\bs M = M_s\left( \coth(\gamma H) - \frac{1}{\gamma H} \right) \frac{\bs H}{H}=\big(\alpha(H)-1\big)\bs H & \text{in }\Omega, \\
\rho(\bs u\cdot\nabla)\bs u - \eta\Delta \bs u + \nabla\tilde p = \bs f & \text{in }\Omega,\\
\nabla\cdot\bs u = 0 & \text{in }\Omega,\\
\psi =\frac{M_s}{\gamma}\ln\frac{\sinh(\gamma H)}{H} = \beta(H) & \text{in }\Omega, \\
p = \tilde p + \mu_0\psi & \text{in }\Omega.
\end{array}
\right.
\end{equation}
\begin{remark}\label{rem2.1}
The system \eqref{eq:eq-sys} is a decoupled system. Firstly, we can solve the first nonlinear elliptic equation of \eqref{eq:eq-sys} to get $\phi$, and solve the Navier-Stokes equations, i.e. the fourth and fifth equations, to get $\bs u$ and $\tilde p$. Secondly, we can get $\bs H$ from the second equation of \eqref{eq:eq-sys}. Finally, we can obtain $\bs M$, $\psi$, and $p$ from the third, the sixth, and the seventh equations,
respectively. 
\end{remark}

\subsection{ Preliminary estimates for nonlinear functions}

Notice that the decoupled system  \eqref{eq:eq-sys} involves the nonlinear functions $\alpha(x)$ and $\beta(x)$ defined in \eqref{eq:alpha-def} and \eqref{eq:beta-def}, respectively.
The following basic estimates of these two functions will be used in  later   analysis.

\begin{lemma}\label{lem:alphaH}
\begin{itemize}
\item[(i)] There exist a positive constants $C_1$ and $C_\alpha$ such that for any $x > 0$, there hold
$$
1 <\alpha(x) \leq C_1, \qquad 
|\alpha^\prime(x)| \leq C_\alpha \ ; 
$$ 
\item[(ii)] For any $x > 0$, there holds
$$
|\beta^{\prime}(x)| \leq M_{s}.
$$
\end{itemize}
\end{lemma}
\begin{proof}
 (i) We first show $
1 <\alpha(x) \leq C_1. $ 
On one hand, the L'Hopital law implies that
$$
\lim\limits_{y \rightarrow 0^+} y\coth(y) = 1,
$$
which, together with the fact that
$$
(y\coth(y))^\prime  = \coth(y) - y\csch^2(y) = \frac{e^{2y} - e^{-2y} - 4y}{(e^y - e^{-y})^2}>0 \quad \forall~~ y > 0,
$$
shows 
$$
y\coth(y) > 1\qquad \forall~~ y > 0.
$$
Therefore,
$$
\alpha(x) = 1 + \gamma M_s\frac{y\coth(y) -1 }{y^2} > 1\qquad\text{with }y = \gamma x,\text{ for }\forall~~x >0.
$$
On the other hand, we easily see  that
$$
\lim\limits_{x\rightarrow 0^+} \alpha(x) = 1+ \frac{\gamma M_s}{3}\qquad\text{and}\qquad\lim\limits_{x\rightarrow + \infty} \alpha(x) = 1,
$$
which, together the continuity of   $\alpha(x)$  for $x>0$, indicate that 
there exists a positive constant $C_1$ such that
$$
\alpha(x) \leq C_1\qquad\forall~~x>0.
$$

Second, let us show $
|\alpha^\prime(x)| \leq C_\alpha.$   It is easy to get
\begin{equation*}
\begin{split}
\alpha^\prime(x) & = \frac{2M_s}{\gamma x^3} - \frac{M_s}{x^2}\coth(\gamma x) - \frac{\gamma M_s}{x}\csch^2(\gamma x) \\
& = \frac{\gamma^2 M_s}{y^3}(2 - y^2\csch^2(y) - y\coth(y))\qquad\qquad\text{with } y = \gamma x.
\end{split}
\end{equation*}
The L'Hopital law implies that
$$
\lim\limits_{x\rightarrow 0^+}\alpha^\prime(x) = 0\qquad\text{and}\qquad \lim\limits_{x\rightarrow +\infty}\alpha^\prime(x) = 0.
$$
Since $\alpha^\prime(x)$ is continuous for $x>0$, we know that
there exists a positive constant $C_\alpha$ such that
$$
|\alpha^\prime(x)| \leq C_\alpha \qquad \forall x>0.
$$
As a result,  the conclusions of (i) follow.

 (ii)  It is easy to find that
$$
\beta^{\prime}(x) = M_{s}\left(\coth(\gamma x) - \frac{1}{\gamma x}\right) = M_{s} \left(\coth(y) - \frac{1}{y}\right) 
$$
with $y = \gamma x>0$. The L'Hopital law implies that
$$
\lim\limits_{x\rightarrow 0^{+}} \beta^{\prime}(x) = 0\quad\text{and}\quad\lim\limits_{x \rightarrow + \infty}\beta^{\prime}(x) = M_{s}.
$$
The fact that 
$$
\left( \coth(y) - \frac{1}{y}\right)^{\prime} = \frac{1-y^{2}\csch^{2}(y)}{y^{2}} = \frac{1+y\csch(y)}{y^{2}}\frac{e^{y} - e^{-y} - 2y}{e^{y} - e^{-y}} > 0 \qquad\forall~~y>0
$$
implies that $\beta^{\prime}(x)$ is a monotonically increasing function on the interval $(0,+\infty)$. We conclude that
$$
0< \beta^{\prime}(x) \leq M_{s}.
$$
This finishes the proof.
\end{proof}

\subsection{Weak formulations}
Based on the FHD model \eqref{eq:eq-sys} and Remark \ref{rem2.1}, we consider the following 
  weak formulations: Find $\phi\in H_0^1(\Omega)$, $\bs H \in\bs H_0(\curl)$, $\bs u \in \bs H_{0}^{1}(\Omega)$ and $\tilde p \in L_0^2(\Omega)$
such that
\begin{align}
\label{eq:weak-maxwell-1}
& a(\phi;\phi,\tau) = -(g,\tau)\qquad\qquad\qquad\qquad\qquad\qquad\qquad\ \,\forall~~\tau \in  H_{0}^1(\Omega),\\
\label{eq:weak-maxwell-2} & (\bs H,\bs C) - (\nabla\phi,\bs C) = 0\quad\qquad\qquad\qquad\qquad\qquad\qquad \forall~~\bs C \in  \bs H_{0}(\curl),\\
\label{eq:weak-NS-1}
& (\rho(\bs u\cdot\nabla)\bs u,\bs v) + \eta(\nabla\bs u,\nabla\bs v) -(\tilde p,\nabla\cdot\bs v)  
 = (\bs f,\bs v)
\qquad\forall~~\bs v \in  \bs H_{0}^{1}(\Omega),\\
\label{eq:weak-NS-2} & (\nabla\cdot\bs u,q) = 0\,\quad\qquad\qquad\qquad\qquad\qquad\qquad\qquad\qquad\forall~~q \in L_{0}^{2}(\Omega), 
\end{align}
where  $a(\cdot;\cdot,\cdot):~H^1(\Omega) \times H^1(\Omega) \times H^1(\Omega) \rightarrow \mathbb R$ is defined by
$$
a(w;\phi,\tau) := \int_{\Omega}\alpha(|\nabla w|) \nabla\phi \cdot\nabla \tau \dd \Omega\qquad\forall~~w,~\phi,~\tau \in H^1(\Omega).
$$
\begin{remark} As shown in Remark \ref{rem2.1}, once the variables $\phi$, $\bs H$, $\bs u$ and $\tilde p$ are solved,   the other variables, i.e.  $\bs M$, $\psi$ and $p$,  can immediately be obtained by the third, the sixth and the seventh equations of \eqref{eq:eq-sys}, respectively.
\end{remark}

In what follows, we discuss the existence and uniqueness of the solutions to the weak formulations \eqref{eq:weak-maxwell-1} - \eqref{eq:weak-NS-2}. 

We first consider 
the nonlinear equation  \eqref{eq:weak-maxwell-1}.  
It is easy to see that, for any given $w \in H^1(\Omega)$, $a(w;\cdot,\cdot)$ is a bilinear form on   $H^1(\Omega)$.  Recall  that  the   Poincar\'e inequality  
\begin{equation*}\label{eq:poin}
\|v\| \leq C_p \|\nabla v\| \qquad\forall~~v \in H_0^1(\Omega),
\end{equation*}
 with $C_p > 0$ being a constant depending only on $\Omega$, means that the semi-norm $\|\nabla (\cdot)\|$ is also a norm  on $H_0^1(\Omega)$. Then, from Lemma \ref{lem:alphaH} (i),
we easily obtain    the following uniform coercivity and continuity results  for $a(w;\cdot,\cdot)$:
\begin{lemma}\label{lem:cocervity}
For any $  w \in H_0^1(\Omega)$ and $\phi, \tau\in H_0^1(\Omega)$, we have
$$
a(w;\phi,\phi) \geq \|\nabla\phi\|^2 
$$
and
$$
a(w;\phi,\tau) \leq C_1 \|\nabla\phi\|\|\nabla \tau\|.
$$
\end{lemma}
 
 We have  the following wellposedness results for equation \eqref{eq:weak-maxwell-1}.
 
 \begin{theorem}\label{well-phi}
The nonlinear equation \eqref{eq:weak-maxwell-1} has at least one solution $\phi\in H_0^1(\Omega)$, and there holds 
\begin{equation}\label{eq:phi-bd}
\|\nabla\phi\| \leq \frac{1}{\mu_0}\|\bs H_e\|.
\end{equation}
Moreover, \eqref{eq:weak-maxwell-1} admits   at most one solution    $\phi\in H_0^1(\Omega)$  satisfying 
\begin{equation}\label{eq:ass-phi}
 \|\nabla\phi\|_{L^\infty} < 1/C_\alpha.
\end{equation}
\end{theorem}

\begin{proof}  The existence of a solution  $\phi\in H_0^1(\Omega)$ follows from Lemma \ref{lem:cocervity} and \cite[Page 332, Theorem 2]{Boccardo1992} directly.

Recalling that $ g=\frac{1}{\mu_0}\div\bs H_e$,
  from \eqref{eq:weak-maxwell-1} and the above lemma  we have 
   $$
\|\nabla\phi\|^2\leq a(\phi;\phi,\phi)=-(g,\phi)=\frac{1}{\mu_0}(\bs H_e, \nabla \phi)\leq \frac{1}{\mu_0}\|\bs H_e\| \| \nabla \phi\|.
$$
This yields the estimate \eqref{eq:phi-bd}.

The thing left is to show the uniqueness of the solution under condition \eqref{eq:ass-phi}.  In fact, let $\phi$ and $\tilde\phi$ be any two solutions of \eqref{eq:weak-maxwell-1} satisfying \eqref{eq:ass-phi}, we have
 $$
a(\tilde\phi;\tilde\phi,\tau) = a(\phi;\phi,\tau)\qquad\forall~~\tau \in H_0^1(\Omega),
$$
which yields 
\begin{align*}
a(\tilde\phi;\tilde\phi-\phi,\tau) = \int_\Omega \big(\alpha(|\nabla\phi|) - \alpha(|\nabla\tilde\phi|) \big)\nabla\phi\cdot\nabla\tau\dd\Omega\qquad\forall~~\tau \in H_0^1(\Omega).
\end{align*}
Taking $\tau = \tilde\phi - \phi$ in the above equation and using Lemma \ref{lem:alphaH}  give
\begin{align*}
\|\nabla(\tilde\phi - \phi)\|^2 & \leq a(\tilde \phi;\tilde \phi-\phi,\tilde\phi - \phi) \\
& = a(\phi;\phi,\tilde\phi - \phi) - a(\tilde\phi;\phi,\tilde\phi - \phi) \\
& \leq C_\alpha\|\nabla \phi\|_{L^\infty}\|\nabla(\tilde\phi -  \phi)\|^2,
\end{align*}
which, together with \eqref{eq:ass-phi}, implies $\tilde \phi = \phi$. This finishes the proof.
\end{proof}

For equation \eqref{eq:weak-maxwell-2}, we have the following conclusion:
\begin{theorem}\label{H-well}
For  any given $\phi \in H_0^1(\Omega)$, equation \eqref{eq:weak-maxwell-2} admits a  unique solution $\bs H = \nabla\phi \in \bs H_0(\curl)$, which   means  $
  \curl \bs H = 0.
  $
\end{theorem}
\begin{proof}
 The de Rham complex \cite{Arnold;Falk;Winther2006} on $2D$ and $3D$ domain implies that $\nabla\phi\in \bs H_0(\curl)$. Thus, taking $\bs C = \bs H - \nabla\phi \in \bs H_0(\curl)$ in \eqref{eq:weak-maxwell-2} yields 
$
\bs H = \nabla\phi. 
$
\end{proof}

Since    \eqref{eq:weak-NS-1}-\eqref{eq:weak-NS-2} are the weak formulations of  the Navier-Stokes equations, the following 
existence and uniqueness results are standard  (cf.  \cite[Page 285-287, Theorems 2.1 and 2.2]{Girault1986}). 

\begin{theorem}
Given $\bs f \in (H^{-1}(\Omega))^d$, there exists at least one pair $(\bs u,\tilde p) \in \bs H_0^1(\Omega) \times L_0^2(\Omega)$ satisfying \eqref{eq:weak-NS-1} - \eqref{eq:weak-NS-2}, and holds
$$
\|\nabla\bs u\| \leq \frac 1\eta \|\bs f\|_{-1},
$$
where $\|\bs f\|_{-1}: = \sup\limits_{\bs v \in \bs H_0^1(\Omega)}\frac{\int_\Omega \bs f \cdot \bs v\dd\Omega}{\|\nabla \bs v\|}.$
In addition, if
$$
(\rho \mathcal N/\eta^2)\|\bs f\|_{-1} <1\quad\text{with}\quad
\mathcal N: = \sup\limits_{\bs u,\bs v,\bs w\in \bs H_0^1(\Omega)}\frac{\int_\Omega (\bs u\cdot\nabla)\bs v \cdot \bs w\dd\Omega}{\|\nabla \bs u\|\|\nabla \bs v\|\|\nabla \bs w\|},
$$
then the solution pair  $(\bs u,\tilde p)$ is unique.
\end{theorem}

\section{Finite element methods}

\subsection{Finite element spaces}
Let $\mathcal T_{h}$ be a quasi-uniform 
 simplicial decomposition  of $\Omega$ with mesh size $h:=\max\limits_{K\in \mathcal T_h} h_K$, where $h_K$ denotes the diameter of $K$  for any $K \in \mathcal T_h$. 
 
 For an integer $l \geq 0$, let $\mathbb P_l(K)$ denote the set of polynomials, defined on $K\in \mathcal T_h$, of degree no more than $l$.
We introduce the following finite element spaces:
\begin{itemize}
  \item  $\bs V_{h} = (V_h)^d \ \big(\subset \text{or } \not\subset (H_{0}^{1}(\Omega))^d\big)$  is the Lagrange element space~\cite{Girault1986,Brezzi;Fortin1991,Girault;Raviart2012} for the velocity $\bs u$, with $\mathbb P_{l+1}(K) \subset  V_h|_K$ for any $K \in \mathcal T_h$.
In particular, for the nonconforming case that $V_h\not\subset H_{0}^{1}(\Omega)$, 
  each $v_h \in V_h$  is required to satisfy the   following conditions: 
\begin{itemize}
\item[(i)]  $v_h$   vanishes at all the nodes on   $\partial \Omega$;
\item[(ii)]  
$|v_h|_{1,h}:=(\sum\limits_{K\in \mathcal T_h} \|\nabla v_h\|_{K}^2)^{1/2}$ is a norm.
 \end{itemize}
 Note that the classical nonconforming  Crouzeix-Raviart (CR) finite element ~\cite{Crouzeix1973CRele}  is corresponding to the nonconforming case with $l=0$. In the nonconforming case, the gradient and divergence operators, $\nabla$ and $\nabla\cdot$, in the finite element scheme \eqref{eq:fem-3} - \eqref{eq:fem-4}  will be  understood as $\nabla_h$ and $\nabla_h\cdot$, respectively, which  denote respectively the piecewise gradient and divergence operators acting on element by element in $\mathcal T_{h}$. 

  \item $W_{h} \subset L_{0}^{2}(\Omega)$ is the piecewise polynomial space 
  for the `pressure' variable 
  $\tilde p$, with $\mathbb P_l(K) \subset W_h|_K$ for any $K \in \mathcal T_h$. 
  \item $S_{h} \subset H_{0}^{1}(\Omega)$ is the Lagrange element space~\cite{Ciarlet1978} 
  for the new variable $\phi$, with $\mathbb P_{l+1}(K)\subset S_h|_K$ for any $K \in \mathcal T_h$.
  \item $\bs U_{h} \subset \bs H_{0}(\curl)$ is the edge element space~\cite{Nedelec1980,Nedelec1986} for the magnetic field $\bs H$, 
  with $(\mathbb P_l(K))^d\subset \bs U_h|_K$ and $(\mathbb P_l(K))^{2d-3} \subset \curl \bs U_h|_K$ 
   for any $K \in \mathcal T_h$.

\end{itemize}
In addition, we make the following assumptions for the above finite element spaces.
\begin{itemize}
  \item[(A1)] There holds the inf-sup condition
  \begin{equation}
\label{eq:inf-sup}
\inf\limits_{\bs v_h \in \bs V_{h}}\frac{(\nabla\cdot\bs v_h,q_h)}{\|\bs v_h\|_{1}} \geq \beta_{0} \|q_h\|\qquad\forall~~q_h \in W_{h},
\end{equation}
where $\beta_{0}>0 $ is a constant independent of $h$;
  \item[ (A2)] The diagram
  \begin{equation}\label{eq:exact-seq}
\begin{CD}
H_0^1@>{\grad}>>  \bs H_0(\curl)  \\
@VV \pi_h^{s} V @ VV \pi ^{c}_h V  \\\
S_{h} @>{\grad}>> \bs U_{h} \end{CD}
\end{equation}
is a commutative sequence. Here $\pi_h^s:~H_0^1(\Omega)\cap C^0(\overline\Omega)\rightarrow S_h$ and $\pi_h^c:~\bs H_0(\curl)\rightarrow \bs U_h$ are the classical interpolation operators, and $\grad$ denotes the gradient operator. Note that the diagram \eqref{eq:exact-seq} also indicates that 
$$ \grad S_{h} \subset \bs U_{h} .$$

\end{itemize}

We recall the following  inverse inequality:
\begin{equation}\label{inverse}
\|\nabla s_h\|_{\infty} \leq C_{inv} h^{-d/2}\|\nabla s_h\| \qquad \forall s_h \in S_h,
\end{equation}
where $  C_{inv} >0$ is a constant   independent of $h$. 

\begin{remark}
There are many finite element spaces that satisfy (A1) and (A2). For example, we can choose $\bs V_h$ and $W_h$ as the Taylor-Hood element pairs. And the spaces $S_h$ and $\bs U_h$ can be respectively chosen as the lowest order Lagrange finite element space and the lowest order N\'ed\'elec edge element space ~\cite{Arnold;Falk;Gopalakrishnan2012,Arnold;Falk;Winther2006,Arnold;Falk;Winther2010}.
\end{remark}
\begin{remark}
In this paper, we consider the conforming finite element spaces for the Navier-Stokes equation for simply of notations. In fact the nonconforming finite element spaces which satisfying the inf-sup condition (A1), with replace the global differential operators to element wise, are also work for the Navier-Stokes equations. For example the $CR$ - $\mathbb P_0$ finite element spaces are work for the Navier-Stokes equations.
\end{remark}

\subsection{Finite element scheme} In view of  \eqref{eq:weak-maxwell-1}-\eqref{eq:weak-NS-2}, 
we consider the following    finite element scheme  for the FHD model: Find $\phi_h \in S_h$, $\bs H_h \in \bs U_h$, 
$\bs u_h \in \bs V_h$ and $\tilde p_h\in W_h$, 
such that  
\begin{align}
\label{eq:fem-1}
& a(\phi_h;\phi_h,\tau_h) = -(g,\tau_h)
&\forall~~\tau_h \in S_h,\\
\label{eq:fem-2}& (\bs H_h,\bs C_h) - (\nabla\phi_h,\bs C_h)  = 0 
&\forall~~\bs C_h \in \bs U_h,\\
\label{eq:fem-3} & \eta(\nabla\bs u_h,\nabla\bs v_h) +b(\bs u_h;\bs u_h,\bs v_h) - (\tilde p_h,\nabla\cdot\bs v_h)  = (\bs f,\bs v_h)
&\forall~~\bs v_h \in \bs V_h, \\
\label{eq:fem-4} & (\nabla\cdot\bs u_h,q_h)  = 0\ 
&\forall~~q_h \in W_h,
\end{align}
where the trilinear form $ b(\cdot;\cdot,\cdot) : \bs H^{1}(\Omega) \times\bs H^{1}(\Omega) \times \bs H^{1}(\Omega) \rightarrow  \mathbb R$ is defined by
$$
b(\bs w;\bs u,\bs v) := \frac{\rho}{2}[((\bs w\cdot\nabla)\bs u,\bs v) - ((\bs w\cdot\nabla)\bs v,\bs u)].
$$


\begin{remark}	\label{rem:sol}
Similar to  \eqref{eq:weak-maxwell-1}-\eqref{eq:weak-NS-2}, the finite element scheme \eqref{eq:fem-1} - \eqref{eq:fem-4} is a decoupled system.  We can first solve the nonlinear equation \eqref{eq:fem-1} to get $\phi_h$, and solve the Navier-Stokes system \eqref{eq:fem-3} - \eqref{eq:fem-4}   to get $\bs u_h$ and $\tilde p_h$. Then we can get $\bs H_h$ from \eqref{eq:fem-2}.   Finally,    we can get   $\bs M_h \in \bs U_h$, the approximation of the magnetization $\bs M$,   from
\begin{equation}\label{eq:weak-maxwell-3-dis}
(\bs M_h,\bs F_h) = \big(
(\alpha(H_h)-1)\bs H_h,\bs F_h \big)
\,\,\,\,\,\qquad \quad \forall~~\bs F_h \in \bs U_h
\end{equation} 
with $H_h:=|\bs H_h|$,   get  $\psi_h \in W_h$, the approximation of $\psi$, from
\begin{equation}\label{eq:fem-5}
(\psi_h,\chi_h) =\big(\beta(H_h),\chi_h\big) 
\ \ \qquad\qquad\qquad\forall~~\chi_h \in W_h,
\end{equation} 
and get  $p_h \in W_h$, the   approximation of the pressure $p$,  from 
\begin{equation}\label{eq:fem-6}
p_h = \tilde p_h + \mu_0 \psi_h.
\end{equation} 
\end{remark}
\begin{remark}\label{rem:trilinear}
It is easy to see that the trilinear form $b(\cdot;\cdot,\cdot)$ is skew-symmetric with respect to the last two variables, i.e., 
\begin{equation}\label{eq:skew-s}
b(\bs w;\bs u ,\bs v ) = -b(\bs w ;\bs v ,\bs u )\qquad \forall ~\bs w ,~\bs u , ~\bs v  \in \bs H^{1}(\Omega) ,
\end{equation}
and that 
\begin{equation}\label{eq:breform}
b(\bs w;\bs u ,\bs v ) =\rho\big((\bs w\cdot\nabla)\bs u, \bs v\big) + \frac{\rho}{2}\big( (\nabla\cdot\bs w) \bs u,\bs v \big)\quad\forall~\bs w,~\bs u \in \bs H^{1}(\Omega) , \ \ \forall  \ \bs v \in \bs H_0^{1}(\Omega).
\end{equation}
As a result, the following two relations hold:
\begin{equation}
\label{eq:b-coer}
  b(\bs w;\bs v,\bs v) =0 \qquad\forall~  \bs w, \bs   v \in \bs H^{1}(\Omega),\\
\end{equation}
\begin{equation}
\label{eq:b-}
b(\bs w;\bs u ,\bs v ) =\rho\big((\bs w\cdot\nabla)\bs u, \bs v \big) \quad\forall~ \bs w, ~\bs u  \in \bs H^{1}(\Omega) \text{ with } \div\bs w = 0,\   \forall \ \bs v \in \bs H_0^{1}(\Omega).   \end{equation}
 \end{remark}

Theorems \ref{lem:solu1-2}-\ref{lem:solu3-4} show  the wellposedness of the finite element scheme \eqref{eq:fem-1} - \eqref{eq:fem-4}.
\begin{theorem}\label{lem:solu1-2}
The nonlinear discrete equation \eqref{eq:fem-1} has at least one solution $\phi_h \in  S_h$,  and there holds 
\begin{equation}\label{eq:nabla-g}
\|\nabla\phi_h \| \leq \frac{1}{\mu_0}\|\bs H_e\|.
\end{equation}
 Furthermore, if the external magnetic filed $\bs H_e$ satisfies 
\begin{equation}\label{eq:ass-He}
\| \bs H_e\| < \mu_0  C_{inv}^{-1}C_\alpha^{-1}  h^{d/2},
\end{equation}
then   \eqref{eq:fem-1} admits a unique solution $\phi_h \in S_h$, and there holds
\begin{equation}\label{eq:phi-inf}
 \|\nabla\phi_h\|_\infty < 1/ C_\alpha.
 \end{equation}
\end{theorem}
\begin{proof}
Define an operator $A:~S_h \rightarrow S_h^\prime$ by
$$
\langle A(\phi_h),\tau_h\rangle = a(\phi_h;\phi_h,\tau_h)\qquad\forall~~\phi_h,~\tau_h\in S_h,
$$
and  define $\Phi:~S_h \rightarrow S_h^\prime$ by 
$$
\Phi(\phi_h) = A(\phi_h) + Q_h^\prime g\qquad\forall~~\phi_h \in S_h,
$$
where $Q_h^\prime:~(H_0^1(\Omega))^\prime \rightarrow S_h^\prime$ is given by
\begin{equation}\label{eq:Q_h}
\langle Q_h^\prime g,\tau_h \rangle = \langle g,\tau_h\rangle=-\frac{1}{\mu_0}(\bs H_e, \nabla \tau_h) \qquad \forall~~\tau_h \in S_h.
\end{equation}
Thus, \eqref{eq:fem-1} is equivalent to the equation
$$
\Phi(\phi_h) = 0.
$$
Lemma \ref{lem:cocervity} implies
$$
\langle A(\phi_h),\phi_h\rangle =a(\phi_h;\phi_h,\phi_h) \geq \|\nabla\phi_h\|^2\qquad\forall~~\phi_h \in S_h.
$$
The definition \eqref{eq:Q_h} and the Cauchy-Schwarz inequality 
imply
$$
\langle Q_h^\prime g,\phi_h\rangle=\langle g ,\phi_h\rangle \leq \frac{1}{\mu_0}\|\bs H_e\|\|\nabla\phi_h\|\qquad\forall~~\phi_h \in S_h.
$$
Therefore,  we have
$$
\langle \Phi(\phi_h),\phi_h\rangle =\langle A(\phi_h) + Q_h^\prime g ,\phi_h\rangle
   \geq 0, \qquad\forall~~\phi_h \in S_h \text{ with }\|\nabla\phi_h\| = \frac{1}{\mu_0}\|\bs H_e\|.
$$
For any given $0 \neq \phi_{h0} \in S_h$ 
and   $\epsilon >0$, denote
  $\delta: = \min\left\{ \frac{\epsilon}{2C_1},~\frac{\epsilon}{2C_\alpha\|\nabla\phi_{h0}\|_\infty}\right\} $. Then for any $\phi_h \in S_h$ satisfying $\|\nabla (\phi_h - \phi_{h0})\| < \delta$, by  Lemmas   \ref{lem:alphaH} and  \ref{lem:cocervity} we have
\begin{align*}
&\|\Phi(\phi_h) - \Phi(\phi_{h0})\|_{S_h^\prime}   = \|A(\phi_h) - A(\phi_{h0})\|_{S_h^\prime} \\
= & \sup\limits_{\tau_h\in S_h} \frac{\langle A(\phi_h) - A(\phi_{h0}),\tau_h \rangle}{\|\nabla\tau_h\|} \\
 = &\sup\limits_{\tau_h \in S_h} \frac{a(\phi_h;\phi_h,\tau_h) - a(\phi_{h0};\phi_{h0},\tau_h)}{\|\nabla\tau_h\|} \\
 =& \sup\limits_{\tau_h \in S_h} \frac{a(\phi_h;\phi_h,\tau_h) - a(\phi_h;\phi_{h0},\tau_h) + a(\phi_h;\phi_{h0},\tau_h) - a(\phi_{h0};\phi_{h0},\tau_h)}{\|\nabla\tau_h\|} \\
 \leq & C_1\|\nabla(\phi_h - \phi_{h0})\| + C_\alpha\|\nabla\phi_{h0}\|_\infty\|\nabla(\phi_h - \phi_{h0})\| \\
 < & \epsilon.
\end{align*}
This means that $\Phi$ is continuous on the set $S_h\setminus\{0\}$.  Notice that Lemma \ref{lem:cocervity} implies $\Phi$ is continuous at the point $0$. Hence, $\Phi$ is continuous on $S_h$. 

On the other hand, the Riesz representation theorem implies that the spaces $S_h$ and $S_h^\prime$ are isometry.  As a result, by \cite[Page 279, Corollary 1.1]{Girault1986} equation \eqref{eq:fem-1} admits at least one solution $\phi_h \in S_h$ satisfiying 
\eqref{eq:nabla-g}.

If $\bs H_e$ satisfies assumption \eqref{eq:ass-He},
then from \eqref{eq:nabla-g} and the inverse inequality \eqref{inverse} we get
$$
  \|\nabla\phi_h\|_\infty \leq    C_{inv}  h^{-d/2}  \|\nabla\phi_h\| < 1/C_\alpha ,
$$
i.e. \eqref{eq:phi-inf} holds. Thus,
by following the same line as in the proof of the uniqueness of the weak solution $\phi$   in Theorem \ref{well-phi}, we can easily obtain the uniqueness of the discrete solution $\phi_h$.
\end{proof}

\begin{theorem}\label{Hh-well}
For  any given $\phi_h \in S_h$, equation \eqref{eq:fem-2} admits a  unique solution $\bs H_h = \nabla\phi_h \in \bs U_h$, which   means  $
  \curl \bs H_h = 0.
  $
\end{theorem}
\begin{proof}
For any given $\phi_h \in S_h$,   assumption (A2) 
on the finite element spaces implies that $\nabla\phi_h \in \bs U_h$. Thus, taking $\bs C_h = \bs H_h - \nabla\phi_h \in \bs U_h$ in \eqref{eq:fem-2} yields
$
\bs H_h = \nabla\phi_h.
$
\end{proof}

The following well-posedness results  for the finite element scheme \eqref{eq:fem-3} - \eqref{eq:fem-4} of  the Navier-Stokes equations are standard (cf. \cite{Girault1986}).

\begin{theorem}\label{lem:solu3-4}
The finite element scheme  
\eqref{eq:fem-3} - \eqref{eq:fem-4} has at least one solution $(\bs u_h,\tilde p_h)\in \bs V_h \times W_h$, and there holds
$$
\|\nabla\bs u_h\| \leq   \frac 1 \eta \|\bs f\|_{-1}.
$$  Further more, if 
\begin{equation}\label{uni-NS}
(\tilde{\mathcal N}/\eta^2)\|\bs f\|_{-1} < 1\quad\text{with}\quad \tilde{\mathcal N} = \sup\limits_{\bs w_h,\bs u_h,\bs v_h \in \bs V_h}\frac{b(\bs w_h;\bs u_h,\bs v_h)}{\|\nabla \bs w_h\|\|\nabla \bs u_h\|\|\nabla \bs v_h\|},
\end{equation}
then the finite element solution $(\bs u_h,\tilde p_h)$ is  unique.
\end{theorem}

\subsection{Error analysis}
In this subsection, we will give   some error estimates   for the finite element scheme \eqref{eq:fem-1} - \eqref{eq:fem-4},  under some rational assumptions.


 Firstly, we have the following error estimate for the discrete solution $\phi_h$.
\begin{theorem}\label{lem:sigma-phi}
Let $\phi \in  H_{0}^{1}(\Omega)\cap  H^{l+2}(\Omega)$ be the solution of \eqref{eq:weak-maxwell-1}, and let $\phi_{h} \in S_{h}$ be the solution of \eqref{eq:fem-1}. Assume that $\phi$ satisfies \eqref{eq:ass-phi}, i.e. $
\|\nabla\phi\|_{L^\infty} < 1/C_\alpha,
$ then there exists a positive constant $C_s$, independent of $h$, such that
 $$
\|\nabla(\phi - \phi_{h})\|  \leq  \frac{C_s(1+C_1) }{1 - C_\alpha\|\nabla\phi\|_{L^\infty}}  h^{l+1} \|\phi\|_{l+2}.
 $$
\end{theorem}
\begin{proof} 
Let   $\phi_h^* \in S_h$ be  the classical interpolation of $\phi$, satisfying the estimate
\begin{equation}\label{phi-interpolation}
\|\nabla(\phi - \phi_h^*)\|\leq C_sh^{l+1}\|\phi\|_{l+2},
\end{equation}
where $C_s>0$ is a  constant independent of $h$ and $\phi$. 

In view of the inequality
$$
\|\nabla(\phi - \phi_h)\| \leq \|\nabla(\phi - \phi_h^*)\| + \|\nabla(\phi_h^* - \phi_h)\|,
$$
it suffices to estimate the term $ \|\nabla(\phi_h^* - \phi_h)\|$. 
To this end, we apply 
Lemma \ref{lem:cocervity} and the relation
$$
a(\phi;\phi,\tau_h) = a(\phi_h;\phi_h,\tau_h),\quad \forall \ \tau_h \in S_h,
$$
to get
\begin{eqnarray*}
 \|\nabla(\phi_h^* - \phi_h)\|^2 & \leq& a(\phi_h;\phi_h^* - \phi_h,\phi_h^* - \phi_h) \\
 &= &a(\phi_h;\phi_h^* - \phi,\phi_h^* - \phi_h) + a(\phi_h;\phi - \phi_h,\phi_h^* - \phi_h).
\end{eqnarray*}
Since
$$
|a(\phi_h;\phi_h^* - \phi,\phi_h^* - \phi_h)| \leq C_1\|\nabla(\phi - \phi_h^*)\|\|\nabla(\phi_h^* - \phi_h)\| 
$$
and
\begin{align*}
a(\phi_h;\phi - \phi_h,\phi_h^* - \phi_h)  =&  a(\phi_h;\phi,\phi_h^* - \phi_h) - a(\phi;\phi,\phi_h^* - \phi_h) \\
 =& \left( \big(\alpha(|\nabla \phi_h|) - \alpha(|\nabla\phi|)\big)\nabla\phi, \nabla(\phi_h^* - \phi_h)\right) \\
 \leq &C_\alpha\|\nabla\phi\|_{L^\infty}\|\nabla(\phi - \phi_h)\|\|\nabla(\phi_h^* - \phi_h)\| \\
\leq & C_\alpha\|\nabla\phi\|_{L^\infty}\big(\|\nabla(\phi - \phi_h^*)\| + \|\nabla(\phi_h^* - \phi_h)\|\big)\|\nabla(\phi_h^* - \phi_h)\|,
\end{align*}
we have 
\begin{eqnarray*}
 \|\nabla(\phi_h^* - \phi_h)\|   \leq  C_1\|\nabla(\phi - \phi_h^*)\|+C_\alpha\|\nabla\phi\|_{L^\infty}\big(\|\nabla(\phi - \phi_h^*)\| + \|\nabla(\phi_h^* - \phi_h)\|\big),
 \end{eqnarray*}
 which, together with \eqref{eq:ass-phi}, implies
\begin{eqnarray*}
 \|\nabla(\phi_h^* - \phi_h)\|   \leq  \frac{C_1 +C_\alpha\|\nabla\phi\|_{L^\infty}}{1 - C_\alpha\|\nabla\phi\|_{L^\infty}}  \|\nabla(\phi - \phi_h^*)\|. 
 \end{eqnarray*}
This inequality, together with  \eqref{phi-interpolation}, indicates the desired conclusion. \end{proof}

In light of Theorems \ref{H-well} and \ref{Hh-well}, we know that
%
$$
\bs H = \nabla\phi,\quad\bs H_{h} = \nabla\phi_{h} \quad \text{and} \quad \curl \bs H=\curl \bs H_{h} =0.
$$
Thus,  by Theorem \ref{lem:sigma-phi}, we immediately get the following error estimate for the discrete solution $\bs H_h$ of \eqref{eq:fem-2}.
\begin{theorem}\label{lem:H}
Let $\bs H \in \bs H_{0}(\curl)$ be the solution of \eqref{eq:weak-maxwell-2}, and let  $\bs H_{h}\in \bs U_{h}$ be the solution of \eqref{eq:fem-2}. Then, under the same conditions as in Theorem \ref{lem:sigma-phi},  
there holds
$$
\|\bs H - \bs H_{h}\|+\|\curl(\bs H - \bs H_{h})\| \leq   \frac{C_s(1+C_1) }{1 - C_\alpha\|\nabla\phi\|_{L^\infty}}  h^{l+1} \|\phi\|_{l+2}.
$$
\end{theorem}

The following error estimate  for the Navier-Stokes system \eqref{eq:fem-3}-\eqref{eq:fem-4} is standard (cf. \cite{Girault1986}).
\begin{theorem}\label{lem:up}
Let $\bs u \in \bs H_{0}^{1}(\Omega) \cap\bs H^{l+2}(\Omega)$ and $\tilde p \in L_{0}^{2}(\Omega)\cap H^{l+1}(\Omega)$ be the solution of the Navier-Stokes system \eqref{eq:weak-NS-1}-\eqref{eq:weak-NS-2}, and let $\bs u_{h} \in \bs V_{h}$ and $\tilde p_{h} \in W_{h}$ be the solution of \eqref{eq:fem-3}-\eqref{eq:fem-4}. There exists a positive constant $C$, independent of $h$,  such that
$$
\|\bs u - \bs u_{h}\|_{1} + \|\tilde p - \tilde p_{h}\| \leq C h^{l+1}(\|\bs u\|_{l+2} + \|\tilde p\|_{l+1}).
$$
\end{theorem}

As shown in Remark \ref{rem:sol}, for the  the magnetization $\bs M=(\alpha(H)-1)\bs H$, 
the new variable $\psi=\beta(H)$, 
and the   pressure   $p=\tilde p + \mu_0 \psi$, 
we can obtain their approximations   $\bs M_h, \psi_h$ and $p_h$   from \eqref{eq:weak-maxwell-3-dis}, \eqref{eq:fem-5} and \eqref{eq:fem-6}, respectively. In view of Theorems \ref{lem:H} and \ref{lem:up},  we easily get the following error estimates for  these  three approximation solutions.
\begin{corollary}\label{lem:M}
 Assume that
$\bs M \in  \bs H^{l+1}(\Omega)$  and  
$\psi \in   H^{l+1}(\Omega)$.  
Under the same conditions as in Theorems \ref{lem:sigma-phi} and \ref{lem:up}, there exists  positive constants $C_M$ and $C_\psi$, independent of $h$,  such that
\begin{align}
&\|\bs M - \bs M_{h}\| \leq C_M h^{l+1}\left( \|\bs M\|_{l+1} +  \frac{(1+C_1) (1+ C_1 + C_\alpha\|\nabla \phi\|_\infty ) }{1 - C_\alpha\|\nabla\phi\|_{L^\infty}}  \|\phi\|_{l+2}\right),
\\
&\|\psi - \psi_{h}\| \leq \left( C_\psi \|\psi\|_{l+1} + \frac{M_sC_s(1+C_1) }{1 - C_\alpha\|\nabla\phi\|_{L^\infty}}   \|\phi\|_{l+2} \right)h^{l+1}, \label{psi-h}
\\
&\|p - p_{h}\| \leq \left( C\|\bs u\|_{l+2} + C\|\tilde p\|_{l+1} + \mu_0C_\psi \|\psi\|_{l+1} + \frac{\mu_0M_sC_s(1+C_1) }{1 - C_\alpha\|\nabla\phi\|_{L^\infty}}   \|\phi\|_{l+2} \right)h^{l+1}. \label{p-h}
\end{align}
\end{corollary}
\begin{proof}
Equation \eqref{eq:weak-maxwell-3-dis} implies that
$$
\bs M_h = Q_h^c((\alpha(H_h) - 1) \bs H_h),
$$
where $Q_h^c:~\bs L^2(\Omega) \rightarrow \bs U_h$ is the $L^2$ orthogonal projection operator. Using the Langevin law \eqref{eq:langevin}, we have
\begin{align*}
\bs M - \bs M_h = \bs M - Q_h^c\bs M + Q_h^c\big((\alpha(H) - 1)\bs H  -  (\alpha(H_h) - 1) \bs H_h\big).
\end{align*}
Thus, by Lemma \ref{lem:alphaH} (i), the boundedness of projection $Q_h^c$, and the relation $\bs H=\nabla \phi$, we get
\begin{align*}
\|\bs M - \bs M_h\| & \leq \|\bs M - Q_h^c\bs M\| + \|(\alpha(H) - \alpha(H_h))\bs H\| + \|(\alpha(H_h)-1)(\bs H - \bs H_h)\|\\
& \leq  \|\bs M - Q_h^c\bs M\|  + (C_\alpha\|\bs H\|_\infty + C_1 + 1)\|\bs H - \bs H_h\|\\
&\leq  \|\bs M - Q_h^c\bs M\|  + (C_\alpha\|\nabla \phi\|_\infty + C_1 + 1)\|\bs H - \bs H_h\|,
\end{align*}
which, together with Theorem \ref{lem:H}  and the standard error estimation   of the projection, gives the desired estimate  for $\bs M_h$.

Similarly, equation \eqref{eq:fem-5} implies that
$$
\bs \psi_h = Q_h^w(\beta(H_h)),
$$
where $Q_h^w:~L^2(\Omega) \rightarrow W_h$ is the $L^2$ orthogonal projection. Then by Lemma \ref{lem:alphaH} (ii) we obtain
\begin{align*}
\|\psi - \psi_h\| &\leq \| \psi - Q_h^w\psi\| + \|Q_h^w\big(\beta(H)  - \beta(H_H)\big)\|\\
 &\leq \| \psi - Q_h^w\psi\| + \|\big(\beta(H)  - \beta(H_h)\big)\|\\
 &\leq \| \psi - Q_h^w\psi\| + M_s\|H - H_h \|\\
 &\leq \| \psi - Q_h^w\psi\| + M_s\|\bs H - \bs H_h \|,
\end{align*}
 which, together with Theorem \ref{lem:H}  and   the projection property, yields the desired estimate \eqref{psi-h}.

Finally,   \eqref{eq:fem-6}  means that 
$$\|p-p_h\|\leq \| \tilde p-\tilde p_h\| + \mu_{0} \|\psi-\psi_h\|, $$
which, together with  Theorem \ref{lem:up} and estimate  \eqref{psi-h}, indicates the desired result \eqref{p-h}.
This completes the proof.
\end{proof}

\section{Numerical experiments}
This section is devoted to three numerical examples to verify the performance of the mixed finite element methods. In all the examples, we solve  the nonlinear system \eqref{eq:fem-1} - \eqref{eq:fem-6}  by using  the {\em i}FEM package~\cite{Chen2009} and Algorithm \ref{alg}. 

\begin{alg}\label{alg}
Given $\phi_h^0\in S_h$ and $\bs u_h^0 \in \bs V_h$,  find $\phi_h \in S_h$, $\bs H_h \in \bs U_h$, $\bs M_h \in \bs U_h$, $\bs u_h \in \bs V_h$, $\tilde p_h \in W_h$, $\psi_h \in W_h$, and $p_h \in W_h$ through five steps:
\begin{enumerate}
\item[Step 1.] For $n = 1,2,\dots,L$ do 
\begin{enumerate}
  \item[(1)] Solving the nonlinear system \eqref{eq:fem-1} as: Find $\phi_h^n \in S_h$ such that
  $$
  a(\phi_h^{n-1};\phi_h^n,\tau_h) = -(g,\tau_h)\qquad\forall~\tau_h \in S_h.
  $$
  \item[(2)] Solving the nonlinear saddle point system \eqref{eq:fem-3}-\eqref{eq:fem-4} as: Find $\bs u_h^n \in \bs V_h$ and $\tilde p_h^n \in W_h$ such that
  \begin{align*}
 & \eta(\nabla\bs u_h^n,\nabla\bs v_h) + b(\bs u_h^{n-1};\bs u_h^n,\bs v_h)  - (\tilde p_h^n,\nabla\cdot\bs v_h) = (\bs f,\bs v_h) 
 & \forall~\bs v_h \in \bs V_h,\\
& (\nabla\cdot\bs u_h^n,q_h) = 0
&\forall~q_h \in W_h.
\end{align*}
\end{enumerate}
\item[Step 2.] Let $\phi_h = \phi_h^L$, $\bs u_h = \bs u_h^L$ and $\tilde p_h = \tilde p_h^L$.
\item[Step 3.] Solving equation \eqref{eq:fem-2}   to get $\bs H_h \in \bs U_h$.
\item[Step 4.] Solving equations  \eqref{eq:weak-maxwell-3-dis}  and \eqref{eq:fem-5} to get $\bs M_h \in \bs U_h$ and   $\psi_h \in W_h$, respectively.
\item[Step 5.] Substituting  $\tilde p_h $ and $\psi_h$   into  \eqref{eq:fem-6} to get   $p_h\in W_h$.
\end{enumerate}
\end{alg}

\begin{remark}
From the convergence theory of   Newton-type methods~\cite{Gil2007,Suli2003}, we know that the iterative solution of Algorithm \ref{alg} will converge to the exact solution, provided that the iteration number $L$ is big enough and the initial guess is nearby the exact solution.  In fact, in all the subsequent numerical examples we choose the  initial guess $\phi_h^0$ as the corresponding  finite element solution of the   Poisson equation 
$$
\Delta \tilde \phi = g  \qquad\text{in }\ \Omega$$
with the same boundary condition as that of the exact solution $\phi$, and choose the initial guess  $\bs u_h^0$   as the  corresponding  finite element solution  of the Stokes equations
$$
\left\{
\begin{array}{ll}
-\eta\Delta\tilde{\bs u} + \nabla p = \bs f  & \text{in }\Omega, \\
\nabla\cdot \tilde{\bs u} = 0 & \text{in }\Omega 
\end{array}
\right.  
$$
with the same boundary condition as that of the exact solution $\bs u$. 
In so doing,   the choice  $L = 2$ works well in the algorithm.
\end{remark}

\begin{figure}[h!]
\centering
\includegraphics[width=0.70\textwidth]{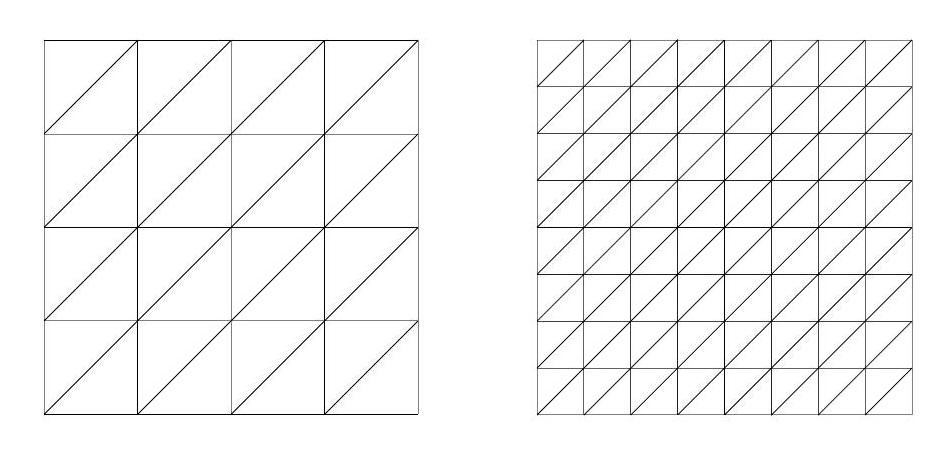}
\caption{The domain $\Omega = (0,1)^2$: $4\times 4$ (left) and $8 \times 8$ (right).}\label{fig:2Dmesh}
\end{figure}

\begin{example} 
\label{ex:2D}
This is a 2D test. We take $\Omega = (0,1)^2$, and use $N \times N$ uniform triangular mesh (cf. Figure \ref{fig:2Dmesh}) with $N = 4,~8,~16,~32,~64,~128$. 
The exact solution of the FHD model \eqref{eq:eq-sys} is given by
\begin{align*}
&\phi(x,y) = (x^2-x)(y^2-y),\\
&\bs H(x,y) = ((2x-1)(y^2-y),(x^2-x)(2y-1))^\intercal,\\
&\bs M(x,y) = M_s\left(\coth(\gamma H) - \frac{1}{\gamma H} \right)\frac{\bs H}{H}, \\
& \bs u(x,y) = (\sin(\pi y),\sin(\pi x))^\intercal,\\
&p(x,y) = 60x^2y - 20y^3 - 5,
\end{align*}
with the parameters $M_s$, $\rho$, $\eta$, $\mu_0$ and $\gamma$ all being   chosen as $1$.

We use the conforming linear ($\mathbb P_1$) element   to discretize the variable $\phi$, the lowest order edge element $NE_0$ to discretize the variables $\bs H$ and $\bs M$, the $CR$ (nonconforming-$\mathbb P_1$) element   to discretize the variables $\bs u$, and the piecewise constant  ($\mathbb P_0$) element   to discretize the variables   $\tilde p$, $\psi$ and $p$.  Note that such a combination of finite element spaces corresponds to $l = 0$, then 
we easily    see from Theorems \ref{lem:sigma-phi} - \ref{lem:up} and Corollary \ref{lem:M} that the theoretical accuracy of the scheme is $\mathcal O(h)$. 
Numerical results are listed in Table \ref{tab:2Dex-1}.
 
\begin{table}[h!]
\footnotesize
\begin{center}
\caption{Numerical results  for Example \ref{ex:2D}.}
\label{tab:2Dex-1}
\begin{tabular}{c|c|c|c|c|c|c}
\hline
\hline  $N$ & $\frac{ \|\nabla(\phi - \phi_{h})\|}{\|\nabla\phi\|} $ &$\frac{ \|\bs H - \bs H_{h}\|_c}{\|\bs H\|_{c}}$ & $\frac{\|\bs M - \bs M_{h}\|}{\|\bs M\|}$ & $\frac{\|\bs u - \bs u_{h}\|_{1,h}}{\|\bs u\|_{1}}$  & $\frac{\|p - p_h\|}{\|p\|}$ & $\|\curl\bs H_h\|_\infty$\\
\hline
$4$   &  $0.3943$  &  $0.3943$  &  $0.3941$  &  $0.7424$ &    $0.3083$ & $0.0003e-12$\\ 
\hline
$8$  &  $0.2023$   &  $0.2023$  &  $0.2023$  &  $0.4385$ &   $0.1524$ & $0.0013e-12$ \\ 
\hline
$16$ &  $0.1018$    &  $0.1018$   &  $0.1018$    &  $0.2352$ &   $0.0717$ & $ 0.0083e-12$\\
\hline
$32$ & $0.0510$   &  $0.0510$   &  $0.0510$    &  $0.1207$ &   $0.0342$ & $ 0.0315e-12$\\
\hline
$64$ & $0.0255$   &  $0.0255$  &  $0.0255$   &  $0.0609$ &    $0.0167$ & $ 0.1485e-12$\\
\hline
$128$& $0.0128$   &  $0.0128$   &  $0.0128$    &  $0.0305$ &    $0.0083$ &  $0.6481e-12 $\\
\hline
order   &  $0.9900$ &  $0.9900$  &  $0.9899$   & $0.9207$ &    $1.0439$ & --\\  
\hline
\hline
\end{tabular}
\end{center}
\end{table}

\end{example}

\begin{figure}[!h]
\centering
\includegraphics[width=0.70\textwidth]{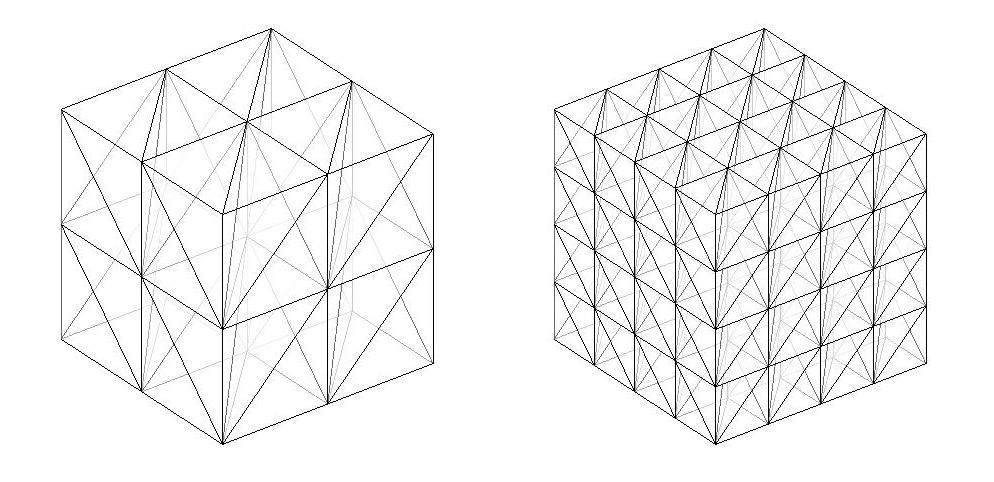}
\caption{The domain $\Omega = (0,1)^3$: $2\times 2\times 2$ (left) and $4\times 4\times 4$ (right).}\label{fig:mesh}
\end{figure}

 The other two experiments, Examples \ref{ex-1} and \ref{ex-2}, are   3D tests, with $\Omega = (0,1)^3$ and use $N \times N \times N$ uniform tetrahedral meshes (cf. Figure \ref{fig:mesh}) with $N = 4,~8,~16$.  
\begin{example} [The lowest order   approximation] 
\label{ex-1}
In this 3D test, the exact solution of the FHD model \eqref{eq:eq-sys} is given by
\begin{align*}
& \phi  = \sin(\pi x)\sin(\pi y)\sin(\pi z), \\ 
&\bs H  = \pi\begin{pmatrix}
\cos(\pi x)\sin(\pi y)\sin(\pi z)\\
\sin(\pi x)\cos(\pi y)\sin(\pi z) \\
\sin(\pi x)\sin(\pi y)\cos(\pi z)
\end{pmatrix}, \\
& \bs M  = M_{s}\left(\coth(\gamma H)- \frac{1}{\gamma H} \right)\frac{\bs H}{H},\\ 
&\bs u  = \begin{pmatrix}
\sin(\pi y), & \sin(\pi z) , & \sin(\pi x)
\end{pmatrix}^\intercal,\\
& p = 120x^2yz-40y^3z-40yz^3,
\end{align*}
with the parameters $M_s$, $\gamma$, $\rho$ and $\eta$ all being choosen as $1$ and $\mu_0 = 10$. 

We use the the conforming $\mathbb P_1$-element to discretize the variable $\phi$, the lowest order N\'ed\'elec finite element space $NE_0$ \cite{Nedelec1980,Nedelec1986} to discretize the variables $\bs H$ and $\bs M$, the nonconforming $CR$ element \cite{Crouzeix1973CRele} to discretize the variable $\bs u$, and  the piecewise constant  $\mathbb P_0$ element to discretize the variables $\tilde p$, $\psi$ and $p$. With these settings,   from Theorems \ref{lem:sigma-phi} - \ref{lem:up} and Corollary \ref{lem:M}  we see that the theoretical accuracy of the scheme is $\mathcal O(h)$.  Numerical results are given in Table \ref{tab:ex-1}.
\begin{table}[h!]
\footnotesize
\begin{center}
\caption{Numerical results  for Example \ref{ex-1}.}
\label{tab:ex-1}
\begin{tabular}{c|c|c|c|c|c|c}
\hline
\hline  $N$ & $\frac{ \|\nabla(\phi - \phi_{h})\|}{\|\nabla\phi\|} $ &$\frac{ \|\bs H - \bs H_{h}\|_c}{\|\bs H\|_{c}}$ & $\frac{\|\bs M - \bs M_{h}\|}{\|\bs M\|}$ & $\frac{\|\bs u - \bs u_{h}\|_{1,h}}{\|\bs u\|_{1}}$  & $\frac{\|p - p_h\|}{\|p\|}$ & $\|\curl\bs H_h\|_\infty$\\
\hline
$4$   &  $0.4739$  &  $0.4739$   &  $0.4755$    &    $0.2066$   &  $0.3665$ & $0.0213e-12$\\ 
\hline
$8$  &  $0.2491$  &  $0.2491 $   &  $0.2506$    &    $0.1082$   &  $0.1803$ & $0.1048e-12$\\ 
\hline
$16$ &  $0.1262$ &  $0.1262$   &   $0.1271$   & 
$0.0551$ & $0.0835$ & $0.6821e-12$\\
\hline
order   &  $0.9545$ &  $0.9545$     &   $0.9515$   &  $0.9529$  &   $1.0671$ & $--$ \\  
\hline
\hline
\end{tabular}
\end{center}
\end{table}


\end{example}

\begin{example}
[A higher order approximation]
\label{ex-2}
In this 3D test, the exact solution of the FHD model \eqref{eq:eq-sys} is given by
\begin{align*}
& \phi(x,y,z)  = (x^2+y^2+z^2)(x^2-x)(y^2-y)(z^2-z), \\
& \bs H (x,y,z) = \begin{pmatrix}
2x(x^2-x)(y^2-y)(z^2-z)+(2x-1)(x^2+y^2+z^2)(y^2-y)(z^2-z) \\
2y(x^2-x)(y^2-y)(z^2-z)+(2y-1)(x^2+y^2+z^2)(x^2-x)(z^2-z) \\
2z(x^2-x)(y^2-y)(z^2-z)+(2z-1)(x^2+y^2+z^2)(x^2-x)(y^2-y)
\end{pmatrix}, \\
& \bs M(x,y,z)  = M_{s}\left(\coth(\gamma H)- \frac{1}{\gamma H} \right)\frac{\bs H}{H},\\
& \bs u(x,y,z)  = \begin{pmatrix}
\sin(\pi y), & \sin(\pi z) , & \sin(\pi x)
\end{pmatrix}^\intercal,\\
& p(x,y,z) = 120x^2yz-40y^3z-40yz^3,
\end{align*}
where the parameters $M_s$, $\gamma$, $\rho$, $\eta$ and $\mu_0$ are all choosen as $1$. 

We use the conforming quadratic  ($\mathbb P_2$) element   to discretize the variable $\phi$, the first order N\'ed\'elec edge element   $NE_1$ \cite{Nedelec1980,Nedelec1986} to discretize the variables $\bs H$ and $\bs M$, the  Taylor-Hood $\mathbb P_2$-$\mathbb P_1$ element to discretize the variables $\bs u$ and $\tilde p$, and the continuous   linear    ($\mathbb P_1$) element to discretize the variables $\psi$ and $p$.  With these settings, we   see 
that the theoretical accuracy of the scheme is $\mathcal O(h^2)$. 
We list   numerical results in Table \ref{tab:ex-2}.
\begin{table}[h!]
\footnotesize
\begin{center}
\caption{Numerical results  for Example \ref{ex-2}.}
\label{tab:ex-2}
\begin{tabular}{c|c|c|c|c|c|c}
\hline
\hline  $N$ & $ \frac{\|\nabla(\phi - \phi_{h})\|}{\|\nabla\phi\|} $ &$\frac{ \|\bs H - \bs H_{h}\|_c}{\|\bs H\|_{c}}$ & $\frac{\|\bs M - \bs M_{h}\|}{\|\bs M\|}$ & $\frac{\|\nabla(\bs u - \bs u_{h})\|}{\|\nabla\bs u\|}$  & $\frac{\|p - p_h\|}{\|p\|}$ & $\|\curl\bs H_h\|$\\
\hline
$4$   &  $0.1369$  &  $0.1369$   &  $0.1369$    &    $0.0450$   &  $0.0535$ & $0.2970e-08$\\ 
\hline
$8$  &  $0.0372$  &  $0.0372$   &  $0.0372$    &    $0.0081$   &  $0.0135$ & $0.7640e-08$\\ 
\hline
$16$ &  $0.0095$ &  $0.0095$   &   $0.0095$   & 
$0.0016$ & $ 0.0034$ & $1.6060e-08$\\
\hline
order   &  $1.9245$ &  $1.9245$     &   $1.9245$   &  $2.4069$  &   $1.9880$ & $--$\\  
\hline
\hline
\end{tabular}
\end{center}
\end{table}

\end{example}

From Examples \ref{ex:2D}, \ref{ex-1} and \ref{ex-2}, we have the following observations:
\begin{itemize}
\item From Tables \ref{tab:2Dex-1} and \ref{tab:ex-1}, we see that the errors in $H^1$ semi-norm for $\phi$, $\bs H(\curl)$ norm for $\bs H$, $L^2$ norm for $\bs M$ and $p$, and discrete $H^1$ norm for $\bs u$, all have the first (optimal) order rate.
\item From Table \ref{ex-2}, we see that the errors in $H^1$ semi-norm for $\phi$ and $\bs u$, $\bs H(\curl)$ norm for $\bs H$, $L^2$ norm for $\bs M$ and $p$, all have the second (optimal) order rate.
\item  From Tables \ref{tab:2Dex-1} - \ref{tab:ex-2}, we can see that the numerical scheme preserves $\curl\bs H_h = 0$ exactly.
\end{itemize}

%





\end{document}